\documentclass[10pt]{amsart}

\usepackage{times,amsfonts,amsmath,amstext,amsbsy,amssymb,
  amsopn,amsthm,upref,eucal, amscd}
\usepackage[T1]{fontenc}
\usepackage{color}

\usepackage[pdftex]{graphicx}

\newtheorem*{thma}{Theorem A}
\newtheorem*{thmb}{Theorem B}
\newtheorem*{corc}{Corollary C}
\newtheorem{theorem}{Theorem}[section]

\newtheorem{proposition}[theorem]{Proposition}

\numberwithin{equation}{section}

\theoremstyle{definition}

\newtheorem{example}[theorem]{Example}

\newtheorem{remark}[theorem]{Remark}

\def\leq{\leqslant }
\def\geq{\geqslant}
 
\def\A{\mathcal{A}}

\begin{document}

\title[Diophantine interval exchange maps]{Cohomological equation and local conjugacy class of Diophantine interval exchange maps}

\author{Giovanni Forni, Stefano Marmi and Carlos Matheus}

\date{\today}

\begin{abstract}
We extend some results of Marmi--Moussa--Yoccoz on the cohomological equations and local conjugacy classes of interval exchange maps of restricted Roth type. 

In particular, we answer a question of Krikorian about the codimension of the local conjugacy class of  self-similar interval exchange maps associated to the Eierlegende Wollmilchsau and the Ornithorynque. 
\end{abstract}
\maketitle

 

\section{Introduction}

The theory of circle diffeomorphisms is a classical subject in Dynamical Systems going back to the seminal works of Poincar\'e and Denjoy. After the groundbreaking results obtained by several mathematicians (including Arnold, Herman and Yoccoz), the cohomological equations and the conjugacy classes of irrational rotations are very well-understood. 

This scenario motivated Forni \cite{Fo97}, \cite{Fo07}, Marmi, Moussa and Yoccoz \cite{MMY05}, and Marmi and Yoccoz \cite{MY16}  to study the cohomological equations 
for interval exchange maps and translation flows (i.e., natural generalizations of circle rotations). 

The linearization results of Marmi, Moussa and Yoccoz \cite{MMY12} cover \emph{almost all}\footnote{In a very precise sense that we are not going to describe here.} interval exchange maps and translation flows. Nevertheless, some natural interesting classes of interval exchange maps were beyond the scope of the works of these authors. In particular, it is not surprising that, during a recent talk by the second author of this paper, Krikorian asked him about the codimension $d_{EW}(r)\in\mathbb{N}$ of the local $C^r$-conjugacy class of a self-similar interval exchange map $T_0$ associated to the \emph{Eierlegende Wollmilchsau}\footnote{An exotic square-tiled surface discovered by Forni in 2006.} (amongst $C^{r+3}$ simple deformations of $T_0$). 
In this article, we answer  Krikorian's question: our Theorem \ref{t.B} implies that $d_{EW}(r) = 4r+8$ for 
all $r\geq 3$.

In fact, in \S\S~\ref{subsec:DC} we introduce classes $DC(\eta,\theta,\sigma)$ of {\it almost Roth} interval exchange transformations (i.e.t.'s) which generalize the class of Roth type i.e.t.'s, introduced by Marmi, Moussa and Yoccoz in~\cite{MMY05}, and we extend the results on cohomological equations of~\cite{MMY05} and the linearization results of~\cite{MMY12} to these new classes, for sufficiently small values of the exponent 
$\eta\geq 0$. The classes $DC(\eta,\theta,\sigma)$ depend on parameters $\eta\geq 0$, $\theta>0$ and 
$\sigma>0$  such that $\eta$ is the analogous of the Diophantine exponent of circle rotations
 (in particular $\eta=0$ for Roth type i.e.t.'s), while the parameters $\theta$ and $\sigma$ are related to the
 ``spectral gap'' and to the (non-uniform) hyperbolicity of the Kontsevich--Zorich cocycle, respectively. Our
 classes of almost Roth type i.e.t.'s generalizes {\it a fortiori} the class of restricted Roth type i.e.t.'s for which
 the linearization theorem of \cite{MMY12} was proved. In particular,  self-similar i.e.t.'s  associated to the \emph{Eierlegende Wollmilchsau} are of Roth type, hence of almost Roth type, but they are \emph{not} of restricted Roth type: thus the linearization results of~\cite{MMY12} do not apply, hence the interest of Krikorian's question.
 
 \smallskip
 In this paper we prove the following results, stated in greater detail in section \ref{section:Statements}.
\begin{thma} [see Theorem~\ref{t.A}] Let $T$ be an i.e.t. on $d$-intervals. There exists 
$\eta_0:= \eta_0 (d, \theta, \sigma)>0$ such that if $T \in DC(\eta,\theta,\sigma)$ for $0\leq \eta <\eta_0$ then the following holds. For any function $\varphi$ piece-wise $C^{r+BV}$, with compact support on each of the subintervals of continuity of $T$, there exists a piece-wise smooth function $\chi$, which is a polynomial of degree $<r$ on every subinterval of continuity, such that the cohomological equation
$$
u \circ T - u = \varphi + \chi
$$
has a solution $u$ of class $C^{r-1}$ on the whole interval of definition of $T$.
  \end{thma}
Theorem A generalizes the results of \cite{MMY05} which were proved for i.e.t.'s of Roth type,
that is, for i.e.t.'s in the class $DC(0, \theta, \sigma)$.   It is an interesting open problem to generalize further our results to arbitrary Diophantine exponents $\eta >0$ (allowing greater loss of regularity  for the solutions).

\begin{thmb} [see Theorem~\ref{t.B}]  Let  $T_0$ be an i.e.t. on $d =2g+ s-1$-intervals, where $g\geq 2$ denote the genus of the corresponding translation surface and $s$ the number of its singularities.  Let $\mu$ denote the dimension of the stable space of the Kontsevich--Zorich cocycle. There exists 
$\eta_0:= \eta_0 (d, \theta, \sigma)>0$ such that if $T_0 \in DC(\eta,\theta,\sigma)$ for $0\leq \eta <\eta_0$ then the following holds.
For each $r\geq 3$ integer, among the $C^{r+3}$ ``simple deformations'' $T$ of $T_0$ (for instance among the non-linear i.e.t's $T$ which coincide with $T_0$ on a neighborhood of its discontinuities) those which are conjugated to $T_0$ by a diffeomorphism $C^r$-close to the identity form a $C^1$-submanifold of codimension $d^*= (g-1)(2r+1)+s+g-\mu$ of the space of all $C^{r+3}$ ``simple deformations'' of $T_0$. 
\end{thmb} 
Theorem B generalizes the results of \cite{MMY12} which were proved for  i.e.t.'s of restricted Roth type
that is, for i.e.t.'s in the class $DC(0, \theta, \sigma)$ such that, in addition, the Kontsevich--Zorich cocycle is 
non-uniformly hyperbolic. It is known that the Kontsevich--Zorich cocycle is non-uniformly hyperbolic with respect to many invariant measures on the moduli space of translation surfaces such as Masur--Veech measures \cite{AV07},~\cite{Fo02},~\cite{Fo11},  measures supported on $SL(2,\mathbb R)$-orbit of non-arithmetic {\it algebraically primitive} Veech surfaces~\cite{Fo11}, and most (in a precise sense) square-tiled surfaces \cite{MMY15}.  However, measures with zero exponents do exist \cite{FMZ11},  \cite{FMZ14}.

\smallskip

As in \cite{MMY12}, section 8,  it is possible to derive from Theorem B an analogous result on ``simple deformations'' of translation flows.  

A translation surface will be called of {\it almost Roth type} 
$DC(\eta,\theta,\sigma)$  if its vertical translation flow is the suspension of an i.e.t. of almost Roth type 
$DC(\eta,\theta,\sigma)$.

\begin{corc}  There exists $\eta_0:=\eta_0 (g,s,\theta, \sigma)>0$ such that the following holds. Given a translation surface of genus $g$ with $s$ cone points of almost Roth type $DC(\eta,\theta,\sigma)$ with $0\leq \eta <\eta_0$ and any integer $r\geq 3$, among the $C^{r+3}$ ``simple deformations''  of the vector field (for instance among small deformations supported away from the set of cone points) those which are conjugated to the vertical flow by a diffeomorphism $C^r$-close to the identity form a $C^1$-submanifold of codimension $d^*= (g-1)(2r+1)+s+g-\mu$ of the space of all $C^{r+3}$ ``simple deformations'' of the vertical flow. 
\end{corc}

\medskip
The proof of Theorem A is based on a refinement of the arguments of \cite{MMY05} and \cite{MMY12}.
The~strategy to solve cohomological equations is to prove that, after a correction of the function by the addition of a piece-wise constant term, its Birkhoff  sums under the i.e.t. are uniformly bounded. A continuous solution then
exists by a version of the Gottschalk-Hedlund theorem. The proof of bounds on Birkhoff sums 
is reduced to bounds on special Birkhoff sums, which correspond to return times of the i.e.t. given by the Rauzy--Veech induction algorithm.  Special Birkhoff sums of functions of bounded variation can be estimated in terms of the matrices generated by the Kontsevich--Zorich cocycle. The expanding directions of the cocycle give obstructions to the desired bounds on special Birkhoff sums. For this reason, special Birkhoff sums are bounded only after an appropriate correction of the function, which eliminates such expanding directions.  The construction of the correction for functions with first derivative of bounded variation is the technical core of the argument.  
A crucial point is that, since  i.e.t.'s commute with the derivative operator on their continuity intervals, it is therefore possible to reduce bounds on Birkhoff sums of the function, up to a piece-wise constant correction, to bounds on Birkhoff sums of its derivative.  Similarly, once the existence of a continuous solution is proved, it is not difficult to extend the result to the existence of smooth solutions. In fact, derivatives of the solution can be constructed by solving the cohomological equation for appropriate corrections of derivatives of the given function by addition of
piece-wise constant terms.

\smallskip
The proof of Theorem B follows along the lines of \cite{MMY12} with virtually no modifications. Our contributions
consists in verifying that the argument carries over from the case of i.e.t.'s of restricted Roth type  to the more
general one of i.e.t.'s of almost Roth type,  once the transversal to the local conjugacy class is modified to take into account the possible neutral directions of the Kontsevich--Zorich cocycle. We argue that in the presence of a neutral Oseledets space of dimension $g-\mu$, the codimension of the local conjugacy class  increases exactly by $g-\mu$, in comparison with the case of i.e.t.'s  of restricted Roth type  (for which $\mu=g$) treated
in \cite{MMY12}. 

We recall the general strategy of the argument given in \cite{MMY12}. Following an idea of M.~Herman, by applying the Schwarzian derivative to the conjugacy equation the linearization problem is reduced to a quasi-linear cohomological equation for the Schwarzian derivative of the conjugacy map (see  \cite{MMY12}). The non-linearity of the equation, on the right hand side of the cohomological equation, depends on the first derivative of the conjugacy and it is small  for non-linear i.e.t.'s which are close to linear in the appropriate topology.  Under the assumption that the (linear) cohomological equation can be solved with a loss of at most $2$ derivatives, since the Schwarzian derivative operator maps the space $C^3$ into $C^0$ and can be inverted, thereby producing a gain of $2$ derivatives, the linearization problem can be solved by the contraction mapping principle 
(or other fixed point theorem for Banach spaces). 

We observe that the above strategy is certainly limited, as in the case of rotations of the circle, to small values of the Diophantine exponent $\eta>0$, since for large values the loss of derivatives in solving the cohomological equation is expected to be greater than $2$, which is the critical loss of derivatives for  Herman's Schwarzian
derivative trick outlined above.

 \medskip
 The paper is organized as follows. In section~\ref{section:Prelim} we recall definitions and basic notions concerning  interval exchange transformations (\S\S \ref{subsec:iet}), Rauzy--Veech algorithm (\S\S \ref{subsec:RV}), Kontsevich--Zorich matrices and special Birkhoff sums (\S\S \ref{subsec:KZandSpecial})  , 
 the Yoccoz acceleration (\S\S \ref{subsec:YA}), function spaces (\S\S \ref{subsec:Funct}) and boundary operators (\S\S \ref{subsec:BO}). As mentioned above, in~\S\S \ref{subsec:DC} we introduce classes of almost Roth type 
 interval exchange transformations. 
 
 In section \ref{section:Statements} we give detailed statements of our main results, Theorem A (\S \S \ref{subsec:CE})  and Theorem B (\S\S \ref{subsec:Linear}) above.  
 Finally, in section \ref{section:ProofA} we explain the proof of Theorem A (in fact, of Theorem~\ref{t.A})
 and in section \label{section:ProofB} the proof of Theorem B (in fact, of Theorem~\ref{t.B}). 
 
 \newpage
\section{Preliminaries}
\label{section:Prelim}

\subsection{Interval exchange maps} \label{subsec:iet} Denote by $\mathcal{A}$ an alphabet on $d\geq 2$ letters. The \emph{interval exchange transformation} (i.e.t.) $T$ associated to two partitions (modulo $0$) 
$$I(T) = \sqcup_{\alpha\in\mathcal{A}} I_{\alpha}^t(T) = 
\sqcup_{\beta\in\mathcal{A}} I_{\beta}^b(T) \quad \textrm{mod. } 0$$ 
of an open bounded interval $I(T)$ into open subintervals $I_{\alpha}^t(T)$ and $I_{\beta}^b(T)$ with $\lambda_{\beta}:=\textrm{length }(I_{\beta}^t(T)) = \textrm{length }(I_{\beta}^b(T))$ for all $\beta\in\mathcal{A}$ is the map sending each $I_{\alpha}^t(T)$ onto $I_{\alpha}^b(T)$ by an appropriate \emph{translation}. The vector $(\lambda_{\alpha})_{\alpha\in\mathcal{A}}$ is called the \emph{length data} of $T$. 

Given an i.e.t. $T$, we write $I(T)=(u_0(T), u_d(T)) = (v_0(T), v_d(T))$ and we denote  
$$\{u_1(T) < \dots < u_{d-1}(T)\} := I(T)\setminus \sqcup_{\alpha\in\mathcal{A}} I_{\alpha}^t(T), \textrm{ resp. }$$
$$\{v_1(T) < \dots < v_{d-1}(T)\} := I(T)\setminus \sqcup_{\beta\in\mathcal{A}} I_{\beta}^b(T),$$
the \emph{singularities} of $T$, resp. $T^{-1}$. 

A \emph{connection} of an i.e.t. $T$ is a triple $(u_k(T), v_l(T), m)$ with $T^m(v_l(T))= u_k(T)$. It was shown by Keane \cite{Ke} that an i.e.t. without connections is minimal. 

The \emph{combinatorial data} attached to an i.e.t. $T$ is a pair of bijections $\pi=(\pi_t, \pi_b)$ from $\mathcal{A}$ to $\{1,\dots, d\}$ given by 
$$(u_{\pi^t(\alpha)-1}(T), u_{\pi^t(\alpha)}(T)) = I_{\alpha}^t(T) \quad \textrm{ and } \quad (v_{\pi^b(\alpha)-1}(T), v_{\pi^b(\alpha)}(T)) = I_{\alpha}^b(T)$$
for all $\alpha\in\mathcal{A}$. 

\begin{remark} In the sequel, we will deal exclusively with \emph{irreducible} combinatorial data: this means that $\pi_t^{-1}(\{1,\dots, k\})\neq \pi_b^{-1}(\{1,\dots, k\})$ for all $0<k<d$.   
\end{remark}

The length and combinatorial data associated to an i.e.t. $T$ can be combined into the complex numbers $\zeta_{\alpha}:=\lambda_{\alpha}+i(\pi_b(\alpha)-\pi_t(\alpha))$ determining the vertices 
$$U_k:=u_0(T) + \sum\limits_{\pi_t(\alpha)\leq k}\zeta_k \quad \textrm{and} \quad V_k = v_0(T) + \sum\limits_{\pi_b(\alpha)\leq k}\zeta_k$$ 
of a polygon whose sides are the $2d$ segments $[U_{k-1}, U_k]$, $[V_{k-1}, V_k]$ for $1\leq k\leq d$. By gluing the sides $[U_{\pi_t(\alpha)-1}, U_{\pi_t(\alpha)}]$ and $[V_{\pi_b(\alpha)-1}, V_{\pi_b(\alpha)}]$ by translation, we obtain a \emph{translation surface} $M_T$. The elements of the subset $\Sigma$ of $M_T$ associated to the vertices $U_k$, $V_k$ of the polygon are called \emph{marked points}. 

The size $d$ of the alphabet $\A$, the genus $g$ of $M_T$, and the cardinality $s$ of the set $\Sigma$ of marked points satisfy 
$$d=2g+s-1$$

Combinatorially, the genus $g$ of $M_T$ can be computed from the rank $2g$ of the anti-symmetric matrix $\Omega=\Omega(\pi)$ given by 
$$\Omega_{\alpha\beta}:=\left\{\begin{array}{cc}1, & \textrm{if } \pi_t(\alpha) < \pi_t(\beta), \pi_b(\alpha) > \pi_b(\beta) \\ 
-1, & \textrm{if } \pi_t(\alpha) > \pi_t(\beta), \pi_b(\alpha)  < \pi_b(\beta) \\ 
0, & \textrm{otherwise}\end{array}\right.$$ 
Geometrically, $\Omega$ is the intersection form of homology cycles in $H_1(M_T\setminus\Sigma, \mathbb{Z})$ represented by simple curves going from the middle of $[V_{\pi_b(\alpha)-1}, V_{\pi_b(\alpha)}]$ towards the middle of $[U_{\pi_t(\alpha)-1}, U_{\pi_t(\alpha)}]$ within the interior of the polygon. 

\subsection{Rauzy--Veech algorithm} \label{subsec:RV} Let $T$ be an i.e.t. with $u_{d-1}(T)\neq v_{d-1}(T)$. The first return map $\widehat{T}$ of $T$ on the interval $\widehat{I}=(u_0(T), \max\{u_{d-1}(T), v_{d-1}(T)\})$ is an i.e.t. whose combinatorial data $\widehat{\pi}$ is naturally described in terms of bijections from $\A$ to $\{1,\dots, d\}$. The operation $T\mapsto \widehat{T}$ is called an \emph{elementary step} of the \emph{Rauzy--Veech algorithm}. 

It is not hard to check that if $T$ has no connections, then $\widehat{T}$ has no connections. In particular, if $T=T(0)$ is an i.e.t. without connections, then we can iterate indefinitely the Rauzy--Veech algorithm in order to obtain a sequence $(T(n))_{n\geq0}$ of i.e.t. with combinatorial data $(\pi(n))_{n\geq0}$ on a sequence of intervals $I(n):=I(T(n))$ sharing the same left endpoint $u_0(T)$. Note that $T(n)$ is the first return map of $T(m)$ to $I(n)$ for all $n\geq m$. 

\subsection{Kontsevich--Zorich matrices and special Birkhoff sums} \label{subsec:KZandSpecial} Fix $m\geq 0$, and define $\alpha_t, \alpha_b\in\A$ by $\pi_t(m)(\alpha_t)=d=\pi_b(m)(\alpha_b)$. The \emph{Kontsevich--Zorich matrix} $B(m,m+1)$ associated to the elementary step $T(m)\mapsto T(m+1)$ of the Rauzy--Veech algorithm is 
$$B(m,m+1) = \left\{\begin{array}{cc} 
\textrm{Id}+E_{\alpha_t \alpha_b}, & \textrm{if } u_{d-1}(T(m)) > v_{d-1}(T(m)) \\ 
\textrm{Id}+E_{\alpha_b \alpha_t}, & \textrm{if } u_{d-1}(T(m)) < v_{d-1}(T(m)) 
\end{array}\right.$$
where $E_{\alpha\beta}$ is the elementary matrix $(E_{\alpha\beta})_{ab}:=\delta_{\alpha a}\delta_{\beta b}$. In general, the Kontsevich--Zorich matrix $B(m,n)\in SL(\mathbb{Z}^{\A})$ for $0\leq m\leq n$ is defined in such a way that the cocycle relation 
$$B(m,p) = B(n,p) B(m,n)$$
holds for all $0\leq m\leq n\leq p$.  

The dynamical interpretation of $B(m,n)$ is the following: the entry $B_{\alpha\beta}(m,n)$ is the number of visits of the $T(m)$-orbit of $I_{\alpha}(n)$ to $I_{\beta}(m)$ before its first return to $I(n)$. In other words, if $\Gamma(T(k))\simeq\mathbb{R}^{\A}$ is the vector space of functions $\varphi$ on $\sqcup_{\alpha\in\A} I^t_{\alpha}(T(k))$ taking a constant value $\varphi_{\alpha}\in\mathbb{R}$ on each $I_{\alpha}^t(T(k))$, then $B(m,n)$ is the matrix of the \emph{special Birkhoff sum} operator $S(m,n):\Gamma(T(m))\to\Gamma(T(n))$ given by 
$$S(m,n)\varphi(x) = \sum\limits_{0\leq k<r(x)}\varphi(T(m)^k(x))$$ 
where $r(x)$ is the first return of $x$ to $I(n)$ with respect to $T(m)$, that is, $r(x)=B_{\alpha}(m,n):=\sum\limits_{\beta\in\A}B_{\alpha\beta}(m,n)$ for $x\in I^t_{\alpha}(T(n))$.  

The Kontsevich--Zorich matrices $B(m,n)$ and the anti-symmetric matrices $\Omega(\pi(k))$ verify 
$$B(m,n)\,\Omega(\pi(m))\, {}^tB(m,n) = \Omega(\pi(n)).$$
In particular, $B(m,n)$ maps $\textrm{Im } \Omega(\pi(m))$ into $\textrm{Im } \Omega(\pi(n))$, and the formulas 
$$\omega(\Omega(\pi(k))v,\Omega(\pi(k))w) :={}^tv\,\Omega(\pi(k))\,w$$ 
provide symplectic structures on $\textrm{Im }(\Omega(k))$ which are respected by $B(m,n)$. Furthermore, ${}^tB(m,n)^{-1}$ maps $\textrm{Ker }\Omega(\pi(m))$ into $\textrm{Ker }\Omega(\pi(n))$, and it is possible to choose a coherent system of basis of $\textrm{Ker }\Omega(\pi(k))$ so that the matrix of ${}^tB(m,n)^{-1}:\textrm{Ker }\Omega(\pi(m))\to \textrm{Ker }\Omega(\pi(n))$ is the identity. 

\subsection{Yoccoz acceleration}  \label{subsec:YA} Define the sequence $(n_k)_{k\in\mathbb{N}}$ as follows: $n_0=0$ and $n_{k+1}$ is the smallest integer such that all entries of $B(n_k, n_{k+1})$ are positive. 

\subsection{Diophantine interval exchange maps} 
\label{subsec:DC}
Let $T$ be an i.e.t. with (irreducible) combinatorial data $\pi$. The space $\Gamma(T)\simeq\mathbb{R}^{\A}$ of functions which are constant on each $I_{\alpha}^t(T)$ contains a hyperplane 
$$\Gamma_0(T) := \left\{\varphi\in\Gamma(T): \int_{I(T)}\varphi(x) \, dx = 0\right\}$$ 
consisting of functions in $\Gamma(T)$ with zero average. 

Given $\eta\geq 0$, $\theta>0$ and $\sigma>0$, we say that an i.e.t. $T$ (without connections) satisfies a Diophantine condition $DC(\eta,\theta,\sigma)$ whenever: 
\begin{itemize}
\item[(a)] For all $\varepsilon>0$, $\|B(n_k, n_{k+1})\|=O_{\varepsilon}(\|B(0,n_k)\|^{\eta+\varepsilon})$; 
\item[(b)] $\|B(0,n_k)|_{\Gamma_0(T)}\|=O(\|B(0,n_k)\|^{1-\theta})$;  
\item[(c)] $B(n_k,n_l)$ induce operators $B_s(n_k,n_l):\Gamma_s(T(n_k))\to\Gamma_s(T(n_l))$ and $B_{\flat}(n_k,n_l): \Gamma(T(n_k))/\Gamma_s(T(n_k))\to\Gamma(T(n_l))/\Gamma_s(T(n_l))$ on the family of stable subspaces $\Gamma_s(T(n_p)):=\{\chi\in\Gamma(T(n_p)):    \|B(n_p,n_q)\chi\|=O(\|B(n_p,n_q)\|^{-\sigma})\,\forall\,q\geq p\}$ 
such that 
$$\|B_s(n_k,n_l)\|=O_{\varepsilon}(\|B(0,n_l)\|^{\varepsilon})=\|B_{\flat}(n_k,n_l)\|$$
for all $\varepsilon>0$;  
\item[(d)] if the dimension $\mu$ of the stable space $\Gamma_s(T)$ is $\mu<g$, then $\limsup\limits_{n\to\infty}\|B(0,n)v\|~>~0$, for all $v\notin\Gamma_s(T)$; 
\end{itemize} 

\begin{remark}\label{r.a} In the case $d=2$, an i.e.t. $T$ is a circle rotation of angle $\alpha$ and, in terms of the convergents $p_k/q_k$ and partial quotients $a_k$ of $\alpha$, the condition (a) above requires that $a_{k+1}=O_{\varepsilon}(q_k^{\eta+\varepsilon})$: in particular, for each $\varepsilon>0$, there exists $\gamma_{\varepsilon}>0$ such that $|\alpha-p/q|\geq \gamma_{\varepsilon}q^{-2-\eta-\varepsilon}$ for all rational numbers $p/q$.  
\end{remark}

\begin{remark}\label{r.b} The quantity $\theta$ in the condition (b) above is related to the second Lyapunov exponent $\lambda_2$ of the Kontsevich--Zorich cocycle acting on the first absolute homology group of the translation surface $M_T$. 
\end{remark}

\begin{remark}\label{r.c} The quantity $\sigma$ in the condition (c) above is related to the least positive Lyapunov exponent of the Kontsevich--Zorich cocycle acting on the first absolute homology group of the translation surface $M_T$. The subexponential growth of $B_s(m,n)$ and $B_{\flat}(m,n)$ in condition (c) above is motivated by Oseledets theorem. 
\end{remark}

\begin{remark}\label{r.d} The stable space $\Gamma_s(T)$ is an isotropic subspace of $\textrm{Im }\Omega(\pi)$. If $\Gamma_s(T)$ is Lagrangian, i.e., $\textrm{dim }\Gamma_s(T) = g$, then the condition (d) is automatic. 
\end{remark}

\begin{example}\label{ex.MMY} An i.e.t. $T$ of \emph{Roth type} in the sense of Marmi--Moussa--Yoccoz \cite{MMY05} (see also \cite{MY16}) satisfy the conditions (a), (b) and (c) above with parameters $\eta=0$, $\theta>0$ and $\sigma>0$. An i.e.t. $T$ of \emph{restricted Roth type} belongs to $DC(0,\theta,\sigma)$ and its stable space $\Gamma_s(T)$ has dimension $g$. 

In particular, a self-similar i.e.t. $T$ associated to a pseudo-Anosov map $f:S \to S$ of a compact orientable surface $S$ is always of Roth type, and, moreover, $T$ is of restricted Roth type only if the induced automorphism $f_*: H_1(S, \mathbb R) \to H_1(S, \mathbb R)$ is hyperbolic.
\end{example}

\begin{example}\label{ex.FMZ11} A self-similar i.e.t. $T$ associated to a pseudo-Anosov map of a \emph{square-tiled cyclic cover} in the sense of \cite{FMZ11} belongs to $DC(0,\theta,\sigma)$: indeed, the conditions (a), (b) and (c) are easy to check, and the condition (d) holds because the Kontsevich--Zorich cocycle on the first absolute homology of $M_T$ acts on its central subspace by isometries of the Hodge norm. 

In particular, a self-similar i.e.t. $T$ associated to a pseudo-Anosov map $f:S\to S$ of the \emph{Eierlegende Wollmilchsau} or the \emph{Ornithorynque} (see \cite{FMZ11}) belongs to $\bigcap\limits_{\theta, \sigma<1} DC(0,\theta,\sigma)$, but $T$ \emph{never} is of restricted Roth type because the vanishing of the non-tautological Kontsevich--Zorich exponents of all Teichm\"uller geodesic flow invariant probability measures supported on the $SL(2,\mathbb{R})$-orbits of the  Eierlegende Wollmilchsau and the Ornithorynque implies that $f_*: H_1(S, \mathbb R) \to H_1(S, \mathbb R)$ is \emph{never} hyperbolic (and, actually, $\textrm{dim }\Gamma_s(T) = 1$).  

\medskip

As it turns out, the self-similar i.e.t.'s from the previous paragraph are very concrete: for instance, if $f:S\to S$ denotes the pseudo-Anosov map of the Eierlegende Wollmilchsau with linear part $\left(\begin{array}{cc} 2 & 1 \\ 1 & 1 \end{array}\right)$ fixing pointwise its conical singularities, then we get a self-similar i.e.t. $T$ associated to $f$ by looking at the first return map of the translation flow in the direction $(\frac{1-\sqrt{5}}{2},1)$ to an appropriate interval inside a separatrix in the direction $(\frac{1+\sqrt{5}}{2},1)$ attached to a conical singularity. 
\end{example}

\subsection{Functional spaces} 
\label{subsec:Funct} 

Given an i.e.t. $T$, let $C^r(\sqcup_{\alpha\in\A}I^t_{\alpha}(T)):=\prod\limits_{\alpha\in\A} C^r(\overline{I^t_{\alpha}(T)})$ and $C^{r+BV}(\sqcup_{\alpha\in\A}I^t_{\alpha}(T)) = \{\varphi\in C^r(\sqcup_{\alpha\in\A}I^t_{\alpha}(T)): D^r\varphi \textrm{ has bounded variation}\}$. 

These functional spaces are our regularity scale for solutions of cohomological equations.  

\subsection{Boundary operator} 
\label{subsec:BO}
Given a combinatorial data $\pi=(\pi_t, \pi_b)$, we define a permutation $\sigma$ of $\A\times \{L,R\}$ by 
$$\sigma(\alpha, R) = \left\{ \begin{array}{ll} (\beta, L) & \textrm{for } \alpha\neq\alpha_t, \pi_t(\beta)=\pi_t(\alpha)+1 \\ (\alpha_b, R) & \textrm{for } \alpha=\alpha_t \end{array} \right.,$$
$$\sigma(\alpha, L)=\left\{ \begin{array}{ll} (\beta, R) & \textrm{for } \alpha\neq{}_b\alpha, \pi_b(\beta)=\pi_b(\alpha)-1 \\ ({}_t\alpha, L) & \textrm{for } \alpha={}_b\alpha \end{array} \right.$$ 
where $\pi_t(\alpha_t)=d=\pi_b(\alpha_b)$ and $\pi_t({}_t\alpha)=1=\pi_b({}_b\alpha)$. The cycles of $\sigma$ form a set denoted by $\Sigma$. 

The \emph{boundary operator} $\partial:C^0(\sqcup_{\alpha\in\A}I^t_{\alpha}(T))\to\mathbb{R}^{\Sigma}$ is 
$$(\partial\varphi)_C := \sum\limits_{c\in C}\varepsilon(c)\varphi(c)\quad \textrm{for } C\in\Sigma$$
where $\varepsilon(\alpha, L):=-1$ and $\varepsilon(\alpha, R):=1$ for all $\alpha\in\A$, and $\varphi(c)$ is the limit of $\varphi$ at the left, resp. right endpoint of $I_{\alpha}^t(T)$ for $c=(\alpha, L)$, resp. $(\alpha, R)$
(see \cite{MMY12}, section 3.1).

In the sequel, we denote by $C^{r+BV}_{\partial}(\sqcup_{\alpha\in\A} I^t_{\alpha}(T)):=\{\varphi\in C^{r+BV}(\sqcup_{\alpha\in\A} I^t_{\alpha}(T)):\partial D^i\varphi=0, \,\, \forall\, 0\leq i<r\}$ and $\Gamma_{\partial}(T)$ the intersection of $\Gamma(T)$ with the kernel of $\partial$. Also, given $T\in DC(\eta,\theta,\sigma)$, let $\Gamma_u$ be a complement of $\Gamma_s(T)$ in $\Gamma_{\partial}(T)$. 

The boundary operator has topological and function theoretic significance. In fact, the subset $\Gamma(T)$ of $C^0(\sqcup_{\alpha\in\A}I^t_{\alpha}(T))$ consisting of all piece-wise constant functions can be identified
with the relative homology space $H_1(M_T, \Sigma, \mathbb R)$.  The restriction of the boundary operator
to this subspace identifies with the boundary operator $$\partial:H_1(M_T, \Sigma, \mathbb R)\to H_0(\Sigma, \mathbb R) \equiv \mathbb R^\Sigma$$ which appears in the relative homology exact sequence.

Thus the kernel of the boundary operator on $\Gamma(T)$ therefore identifies to the absolute homology $H_1(M, \mathbb R)$.  As a consequence, by a  restriction to the kernel of the boundary operator it is possible to exclude the zero Lyapunov exponents coming from the relative part of the Kontsevich--Zorich cocycle.

The kernel of the boundary operator on $C^0(\sqcup_{\alpha\in\A}I^t_{\alpha}(T))$ coincides
with functions which can be obtained by integration of continuous functions on $M_T$ along first return orbit
of the vertical flow (see \cite{MMY12}, section 8.2).

\begin{remark}\label{r.39} If $T\in DC(\eta,\theta,\sigma)$, then $\Gamma_s(T)$ coincides with the space $\Gamma_T$ of $\chi\in\Gamma(T)$ of the form $\chi=\phi\circ T-\phi$ for $\phi\in C^0(\overline{I(T)})$. Indeed, $\Gamma_s(T)\subset \Gamma_T$ and $\Gamma_T$ is contained in the space $\Gamma_{ws}(T)$ of vectors converging to zero under the Kontsevich--Zorich cocycle. Since the Kontsevich--Zorich matrices act trivially on $\Gamma(T)/\Gamma_{\partial}(T)$ and the condition (d) above implies that no vector outside $\Gamma_s(T)$ converges to zero under the Kontsevich--Zorich cocycle, we obtain that $\Gamma_s(T)=\Gamma_T=\Gamma_{ws}(T)$. 
\end{remark}

\section{Statement of the main results}
\label{section:Statements}

\subsection{Cohomological equation} 
\label{subsec:CE}

Given $r\geq 1$, let $\Gamma(r)$ be the $rd$-dimensional space of $\chi\in C^{\infty}(\sqcup_{\alpha\in\A}I^t_{\alpha}(T))$ such that each $\chi|_{I^t_{\alpha}(T)}$ is a polynomial of degree $<r$, let $\Gamma_{\partial}(r)$ be the subspace of $\chi\in\Gamma(r)$ with $\partial D^i\chi=0$ for all $0\leq i<r$, and let $\Gamma_T(r)$ be the subspace of $\chi\in\Gamma(r)$ of the form $\chi = \psi\circ T-\psi$ with $\psi\in C^{r-1}(\overline{I(T)})$.

\begin{theorem}\label{t.A} Let $T$ be an i.e.t. in $DC(\eta,\theta,\sigma)$ for some parameters $\eta\geq0$, $\theta, \sigma>0$ with 
$$\left(4d+\frac{d(d-1)}{\sigma}\right)\eta<\theta.$$
Then, there are bounded operators $L_1(r):C^{r+BV}_{\partial}(\sqcup_{\alpha\in\A} I^t_{\alpha}(T))\to\Gamma_{\partial}(r)/\Gamma_T(r)$ extending the projection $\Gamma_{\partial}(r)\to\Gamma_{\partial}(r)/\Gamma_T(r)$, and $L_0(r):C^{r+BV}_{\partial}(\sqcup_{\alpha\in\A} I^t_{\alpha}(T))\to C^{r-1}(\overline{I(T)})$ such that 
$$\varphi = L_1(r)(\varphi)+ L_0(r)(\varphi)\circ T - L_0(r)(\varphi)$$
\end{theorem}

\begin{remark} Similarly to \cite{MMY12} and \cite{MY16}, one can modify the proof of Theorem \ref{t.A} in order to replace $BV$ regularity by H\"older regularity and improve a little bit the loss of $1+BV$ derivatives in the statement above. Nevertheless, we do not pursue this direction here.  
\end{remark} 

\begin{remark}\label{r.large-eta} It is an interesting problem to extend the statement of Theorem \ref{t.A} to the case of larger parameters $\eta$ (while allowing a higher loss of regularity). 
\end{remark}

\subsection{Linearization of interval exchange maps}
\label{subsec:Linear}
 The solutions of the cohomological equation constructed in Theorem \ref{t.A} together with \emph{Herman's Schwarzian derivative trick} allow to describe the local conjugacy class of any almost Roth i.e.t.:

\begin{theorem}\label{t.B} Let $T_0$ be an almost Roth i.e.t. in the sense that $T_0\in DC(\eta,\theta,\sigma)$ for some parameters $\eta\geq0$, $\theta, \sigma>0$ with 
$$\left(4d+\frac{d(d-1)}{\sigma}\right)\eta<\theta.$$
Then, for each $r\geq 3$ integer, the $C^{r+3}$ simple deformations\footnote{I.e., $T$ is a generalized i.e.t. with the same combinatorial data and singularities of $T_0$ such that $T$ coincides with $T_0$ near the $u_i(T_0)$, $0\leq i\leq d$, and $T$ is a $C^{r+3}$-diffeomorphism on each $I^t_{\alpha}(T_0)$, $\alpha\in\A$.} $T$ of $T_0$ conjugated to $T_0$ by a diffeomorphism $C^r$-close to the identity form a $C^1$-submanifold of codimension $d^*= (g-1)(2r+1)+s+g-\mu$ of the space of all $C^{r+3}$ simple deformations of $T_0$. 
\end{theorem}

\begin{remark} Our codimension $d^*$ coincides with the expression in \cite{MMY12} when $\mu=g$ and, in general, it differs from the previous case only by the quantity $g-\mu$, i.e., half of the total number of zero exponents of the Kontsevich--Zorich cocycle. 
\end{remark} 

\begin{remark} One can modify the arguments in Section 7 of \cite{MMY12} in order to include the case $r=2$ in the statement above, but we are not going to pursue this matter here. 
\end{remark}

\begin{remark} It would be desirable to obtain local linearization results for larger values of $\eta$, but this seems currently out of reach (partly because we can't treat the cohomological equations in this case, cf. Remark \ref{r.large-eta}). 
\end{remark}

\section{Proof of Theorem \ref{t.A}} \label{section:ProofA}

In this section, we revisit the arguments in \cite{MMY05} and \cite{MMY12} in order to refine their results on the cohomological equations. 

\subsection{The i.e.t. $T(n_k)$ is almost balanced} The proposition below generalizes the proposition in page 835 of \cite{MMY05}:

\begin{proposition}\label{p.balanced-times} For all $n\geq 0$, we have 
$$\max\limits_{\alpha\in\A} \lambda_{\alpha}(T(n)) \geq \left(\sum\limits_{\alpha\in\A}\lambda_{\alpha}(T(0))\right) \|B(0,n)\|^{-1} \geq \min\limits_{\alpha\in\A} \lambda_{\alpha}(T(n)).$$

Moreover, the condition (a) implies that 
$$\max\limits_{\alpha\in\A} \lambda_{\alpha}(T(n_k)) = O_{\varepsilon}\left(\|B(0,n_k)\|^{\eta+\varepsilon}\min\limits_{\alpha\in\A} \lambda_{\alpha}(T(n_k))\right)$$ 
for all $\varepsilon>0$ and $k\in\mathbb{N}$. 
\end{proposition}

\begin{proof} These estimates follow from the facts that the lengths $\lambda_{\alpha}(T(m))$ of the intervals exchanged by $T(n)$ are deduced from the lengths $\lambda_{\alpha}(T(n))$ via the matrix $B(m,n)$, and the positivity of the matrices $B(n_k, n_{k+1})$.  
\end{proof}

\subsection{Reducing general Birkhoff sums to special Birkhoff sums} 

Let $\varphi:I(T)\to\mathbb{R}$ be a bounded function, $x\in I(T)$ and $N\in\mathbb{N}^*$. Denote by $y$ the closest point to $u_0(T)$ of the finite orbit $(T^j(x))_{0\leq j<N}$, and let $k\in\mathbb{N}$ be the largest integer such that $y\in I(T(n_k))$. 

It is shown in page 840 of \cite{MMY05} (see also page 132 of \cite{MY16}) that 
$$|S_N\varphi(x)|:=\left|\sum\limits_{j=0}^{N-1}\varphi(T^j(x))\right|\leq \sum\limits_{l=0}^{k}\|B(n_l,n_{l+1})\|\cdot\|S(0,n_l)\varphi\|_{L^{\infty}}$$

This estimate has the following consequence: 

\begin{proposition}\label{p.reduction} If $T$ is an i.e.t. satisfying the condition (a) and $\varphi:I(T)\to\mathbb{R}$ is a bounded function such that, for some constants $C(\varphi)>0$ and $\omega>\eta$, one has 
$$\|S(0,n_k)\varphi\|_{L^{\infty}}\leq C(\varphi) \|B(0,n_k)\|^{-\omega} \quad \forall\,k\in\mathbb{N},$$ 
then the Birkhoff sums of $\varphi$ are bounded. 
\end{proposition} 

\subsection{An application of Gottschalk--Hedlund theorem} 

It is explained in Subsection 3.2 of \cite{MMY12} that Proposition \ref{p.reduction} above and Gottschalk--Hedlund theorem imply the following criterion for the continuity of solutions of the cohomological equation: 

\begin{proposition}\label{p.continuity} Let $T$ be an i.e.t. verifying the condition (a) and $\varphi\in C^0(\sqcup_{\alpha\in\A} I^t_{\alpha}(T))$ such that, for some constants $C(\varphi)>0$ and $\omega>\eta$, one has 
$$\|S(0,n_k)\varphi\|_{L^{\infty}}\leq C(\varphi) \|B(0,n_k)\|^{-\omega} \quad \forall\,k\in\mathbb{N}.$$

Then, there exists $\psi\in C^0(\overline{I(T)})$ with $\varphi = \psi\circ T - \psi$. 
\end{proposition}

\subsection{$C^0$ solutions to cohomological equations}\label{ss.aR}

At this point, we are ready to discuss the following improvement of Theorem 3.10 of \cite{MMY12}: 

\begin{theorem}\label{t.310} Let $T$ be an i.e.t. in $DC(\eta,\theta,\sigma)$ for some parameters $\eta\geq0$, $\theta, \sigma>0$ with 
$$\left(4d+\frac{d(d-1)}{\sigma}\right)\eta<\theta$$
Then, there are bounded operators $L_0:C^{1+BV}_{\partial}(\sqcup_{\alpha\in\A} I^t_{\alpha}(T))\to C^0(\overline{I(T)})$ and $L_1:C^{1+BV}_{\partial}(\sqcup_{\alpha\in\A} I^t_{\alpha}(T))\to\Gamma_u(T)$ such that 
$$\varphi = L_1(\varphi)+ L_0(\varphi)\circ T - L_0(\varphi)$$ 
\end{theorem} 

\begin{proof} Let $\phi\in BV(\sqcup_{\alpha\in\A} I^t_{\alpha}(T))$ with zero mean. As it is explained in page 841 of \cite{MMY05}, one has 
$$\|S(0,n_k)\phi\|_{L^{\infty}}\leq \left(\sum\limits_{0\leq j < k} \|B(n_j, n_{j+1})\|\cdot\|B(n_{j+1}, n_k)|_{\Gamma_0(T(n_{j+1}))}\|\right)\|\phi\|_{BV}$$ 

By the condition (a), we have $\|B(n_j,n_{j+1})\|=O_{\varepsilon}(\|B(0,n_j)\|^{\eta+\varepsilon})$ for all $\varepsilon>0$ and $j\in\mathbb{N}$. 

Next, we estimate $\|B(n_{j+1}, n_k)|_{\Gamma_0(T(n_{j+1}))}\|$ by distinguishing two regimes depending on a parameter $\rho>0$ to be chosen later: 
\begin{itemize}
\item if $\|B(0,n_{j+1})\|\leq \|B(0,n_k)\|^{\rho}$, then $\|B(0,n_{j+1})^{-1}\|=O(\|B(0,n_k)\|^{(d-1)\rho})$ (because $B(m,n)\in SL(\mathbb{Z}^{\A})$); since $B(0,n_k)\cdot B(0,n_{j+1})^{-1}=B(n_{j+1}, n_k)$ and $B(0,n_{j+1})^{-1} \Gamma_0(T(n_{j+1}))=\Gamma_0(T)$, the condition (b) implies that 
$$\|B(n_{j+1}, n_k)|_{\Gamma_0(T(n_{j+1}))}\|=O(\|B(0,n_k)\|^{1-\theta+(d-1)\rho})$$
\item if $\|B(0,n_{j+1})\|>\|B(0,n_k)\|^{\rho}$, then we write $B(0,n_k)=B(n_{j+1}, n_k)\cdot B(0,n_{j+1})$; since the entries of $B(n_{j+1}, n_k)$ are positive when $j+1<k$ and $B(n_{j+1},n_k)=\textrm{Id}$ when $j+1=k$, we obtain 
$$\|B(0,n_k)\|\geq \|B(n_{j+1}, n_k)\|\cdot\|B(0,n_{j+1})\|\geq \|B(n_{j+1},n_k)\|\cdot\|B(0,n_k)\|^{\rho}$$
and, \emph{a fortiori}, 
$$\|B(n_{j+1}, n_k)|_{\Gamma_0(T(n_{j+1}))}\|=O(\|B(0,n_k)\|^{1-\rho})$$
\end{itemize}
By optimizing the value of $\rho>0$ (i.e., taking $1-\theta+(d-1)\rho = 1-\rho$), we deduce that 
$$\|B(n_{j+1}, n_k)|_{\Gamma_0(T(n_{j+1}))}\|=O(\|B(0,n_k)\|^{1-\frac{\theta}{d}})$$
in both regimes. 

From this discussion, we see that 
$$\|S(0,n_k)\phi\|_{L^{\infty}}=O_{\varepsilon}(\|B(0, n_k)\|^{1-\frac{\theta}{d}+\eta+\varepsilon})\|\phi\|_{BV}$$

Now, let $\varphi\in C^{1+BV}(\sqcup_{\alpha\in\A} I^t_{\alpha}(T))$ such that $\int_{I(T)}D\varphi = 0$. As it is explained in Remark 3.11 of \cite{MMY12}, the proof of the theorem is reduced to show that there exists $\chi\in\Gamma(T)$ such that $\varphi-\chi$ satisfies the hypothesis of Proposition \ref{p.continuity}. 

For this sake, given $\phi\in BV(\sqcup_{\alpha\in\A} I^t_{\alpha}(T(n_k)))$ with zero mean, let $P_0^{(k)}\phi$ be the class modulo $\Gamma_s(T(n_k))$ of the primitive of $\phi$ with zero mean on each interval $I^t_{\alpha}(T(n_k))$, $\alpha\in\A$. By following the arguments in pages 843 of \cite{MMY05}, we shall modify $P_0^{(k)}$ into a bounded operator $P^{(k)}=P_0^{(k)}+\Delta P^{(k)}$ such that 
$$S(n_k, n_l)\circ P^{(k)} = P^{(l)}\circ S(n_k, n_l).$$ 
In this direction, we set $\Lambda(k,l) := P_0^{(l)}\circ S(n_k,n_l) - S(n_k,n_l)\circ P_0^{(k)}$ and we will prove that 
$$\Delta P^{(k)} = \sum\limits_{l>k} B_{\flat}(n_k,n_l)^{-1}\circ\Lambda(n_{l-1},n_l)\circ S(n_k,n_{l-1})$$  
is a bounded operator such that $\varphi-\chi:=P^{(0)}D\varphi$ has the desired properties. 

For this sake, recall from page 843 of \cite{MMY05} that 
$$\|\Lambda(n_{l-1}, n_l)\phi\|_{L^{\infty}}\leq 2\|B(n_{l-1}, n_l)\|\cdot\max\limits_{\alpha\in\A} \lambda_{\alpha}(T(n_{l-1}))\cdot \|\phi\|_{L^{\infty}}$$ 
Since the condition (a) and Proposition \ref{p.balanced-times} imply that 
$$\|B(n_{l-1}, n_l)\|=O_{\varepsilon}(\|B(0,n_{l-1})\|^{\eta+\varepsilon})$$
and 
$$\max\limits_{\alpha\in\A} \lambda_{\alpha}(T(n_{l-1}))=O_{\varepsilon}(\|B(0,n_{l-1})\|^{-1+\eta+\varepsilon}),$$ 
we obtain 
$$\|\Lambda(n_{l-1}, n_l)\phi\|_{L^{\infty}}=O_{\varepsilon}(\|B(0,n_{l-1})\|^{-1+2\eta+\varepsilon})$$ 
On the other hand, the condition (c) says that $\|B_{\flat}(0,n_l)^{-1}\|=O_{\varepsilon}(\|B(0,n_l)\|^{\varepsilon})$ and we saw above that $\|S(0,n_l)\phi\|_{L^{\infty}} = O_{\varepsilon}(\|B(0, n_l)\|^{1-\frac{\theta}{d}+\eta+\varepsilon})\|\phi\|_{BV}$. Thus, 
$$\|\Delta P^{(0)}\phi\| = O_{\varepsilon}\left(\sum\limits_{l>0}\|B(0,n_l)\|^{3\eta-\frac{\theta}{d}+\varepsilon}\right)\cdot\|\phi\|_{BV},$$
and, in general, 
\begin{eqnarray*}
\|\Delta P^{(k)}S(0,n_k)\phi\| &=& O_{\varepsilon}\left(\sum\limits_{l>k}\|B(0,n_l)\|^{3\eta-\frac{\theta}{d}+\varepsilon}\right)\cdot\|\phi\|_{BV} \\ &=& O_{\varepsilon}\left(\|B(0,n_k)\|^{3\eta-\frac{\theta}{d}+\varepsilon}\right)\cdot\|\phi\|_{BV}
\end{eqnarray*}
for all $k\in\mathbb{N}$.  

We affirm that the special Birkhoff sums of $P^{(0)} D\varphi$ satisfy the assumptions of Proposition \ref{p.continuity}. Indeed, the definition of $P^{(k)}$ says that the class of $S(0,n_k)P^{(0)}D\varphi$ modulo $\Gamma_s(T(n_k))$ is $P^{(k)}S(0,n_k)D\varphi$, and the definition of $P_0^{(k)}$ and our discussion above imply 
\begin{eqnarray*}
\|P_0^{(k)}S(0,n_k)D\varphi\| &\leq& \max\limits_{\alpha\in\A} \lambda_{\alpha}(T(n_k))\cdot \|S(0,n_k)D\varphi\|_{L^{\infty}} \\ 
&=& O_{\varepsilon}(\|B(0, n_k)\|^{2\eta-\frac{\theta}{d}+\varepsilon}) \|D\varphi\|_{BV}. 
\end{eqnarray*} 
By combining this information with the previous estimate for $\Delta P^{(k)}$, we get 
$$\|S(0,n_k)P^{(0)}D\varphi\| = \|P^{(k)} S(0,n_k)D\varphi\| = O_{\varepsilon}\left(\|B(0,n_k)\|^{3\eta-\frac{\theta}{d}+\varepsilon}\right)\cdot\|D\varphi\|_{BV}$$ 

By definition, this means that $S(0,n_k)P^{(0)}D\varphi = \Phi_k+\chi_k$ with $\chi_k\in\Gamma_s(T(n_k))$ and 
$$\|\Phi_k\|_{L^{\infty}}=O_{\varepsilon}\left(\|B(0,n_k)\|^{3\eta-\frac{\theta}{d}+\varepsilon}\right)\cdot\|D\varphi\|_{BV}$$ 

Set $\Delta\chi_{k+1}:=S(n_k, n_{k+1})\Phi_k-\Phi_{k+1}$, so that $\chi_{k+1}=B(n_k,n_{k+1})\chi_k+\Delta\chi_{k+1}$ and 
$$S(0,n_k)P^{(0)}D\varphi = \Phi_k+\sum\limits_{j\leq k} B(n_j, n_k)\Delta\chi_j.$$
Note that $\|\Delta\chi_j\|=O_{\varepsilon}\left(\|B(0,n_j)\|^{4\eta-\frac{\theta}{d}+\varepsilon}\right)\cdot\|D\varphi\|_{BV}$ by the condition (a), and recall that the condition (c) ensures that $\|B(0,n_j)|_{\Gamma_s(T)}\|=O(\|B(0,n_j)\|^{-\sigma})$. We estimate $B(n_j,n_k)\Delta\chi_j$ depending on two regimes driven by a parameter $\kappa>0$ to be chosen in a moment:
\begin{itemize}
\item if $\|B(0,n_j)\|\leq \|B(0,n_k)\|^{\kappa}$, we write $B(n_j,n_k)=B(0,n_k)\cdot B(0,n_j)^{-1}$ and recall that $B(0,n_j)^{-1}\Gamma_s(T(n_j)) =\Gamma_s(T)$ in order to get 
\begin{eqnarray*}
\|B(n_j,n_k)\Delta\chi_j\|&\leq& \|B(0,n_k)\|^{-\sigma} \|B(0,n_j)\|^{d-1}\|\Delta\chi_j\| \\ 
&=&O_{\varepsilon}(\|B(0,n_k)\|^{-\sigma+(d-1+4\eta-\frac{\theta}{d})\kappa +\varepsilon})\cdot\|D\varphi\|_{BV}
\end{eqnarray*}
\item if $\|B(0,n_j)\| > \|B(0,n_k)\|^{\kappa}$, we use the condition (c) in order to obtain 
\begin{eqnarray*}
\|B(n_j,n_k)\Delta\chi_j\|&=&O_{\varepsilon}(\|B(0,n_k)\|^{\varepsilon}) \|\Delta\chi_j\| \\ 
&=&O_{\varepsilon}(\|B(0,n_k)\|^{(4\eta-\frac{\theta}{d})\kappa+\varepsilon})\cdot\|D\varphi\|_{BV}
\end{eqnarray*}
\end{itemize}
By optimizing $\kappa>0$ (i.e., by choosing $-\sigma+(4\eta-\frac{\theta}{d}+d-1)\kappa = (4\eta-\frac{\theta}{d})\kappa$), we conclude that 
$$\|B(n_j,n_k)\Delta\chi_j\|=O_{\varepsilon}(\|B(0,n_k)\|^{(4\eta-\frac{\theta}{d})\frac{\sigma}{d-1}})\cdot\|D\varphi\|_{BV}$$

Therefore, we have 
$$\|S(0,n_k)P^{(0)}D\varphi\|_{L^{\infty}} = O_{\varepsilon}(\|B(0,n_k)\|^{(4\eta-\frac{\theta}{d})\frac{\sigma}{d-1}})\cdot\|D\varphi\|_{BV},$$ 
so that the special Birkhoff sums of $P^{(0)}D\varphi$ match the hypothesis of Proposition \ref{p.continuity} whenever $(4\eta-\frac{\theta}{d})\frac{\sigma}{d-1}<-\eta$, that is, $(4+\frac{(d-1)}{\sigma})d\eta<\theta$.  
\end{proof}

\subsection{$C^{r-1}$ solutions to cohomological equations}\label{ss.Roth} 

By Proposition 3.12 of \cite{MMY12} (and Remark \ref{r.39} above), we have that 
$$\textrm{dim } \Gamma(r) = rd, \, \textrm{dim } \Gamma_{\partial}(r) = (2g-1)r+1 \, \textrm{ and } \, \textrm{dim } \Gamma_T(r) = \textrm{dim } \Gamma_T + r-1 = \mu+r-1$$

In particular, $\Gamma_{\partial}(r+1)/\Gamma_T(r+1)$ has dimension $(g-1)(2r+1) + g-\mu+1$ and $\Gamma_{\partial}(r-2)/\Gamma_T(r-2)$ has dimension $(g-1)(2r-4) + 2-\mu$. 

As it is explained in the short proof of Theorem 3.13 of \cite{MMY12}, we can use Theorem \ref{t.310} above as the initial step $r=1$ of an inductive argument on $r$ showing Theorem \ref{t.A}.  

\section{Proof of Theorem \ref{t.B}} \label{section:ProofB}

In this section, we provide a slight improvement of the results in \cite{MMY12} about local conjugacy classes of interval exchange maps. 

\subsection{Conjugacy invariants} Let $T$ be a $C^k$ generalized i.e.t. with combinatorial data $\pi$. Given $c\in\A\times\{L,R\}$, let $j(T,c)$ be the $k$-jet at $0$ of 
$$x\mapsto T(u_{\alpha}^t+x)-u^b_{\alpha}$$ 
where $x=0+$, resp. $0-$ for $c=(\alpha, R)$, resp. $c=(\alpha, L)$. 

For each cycle $C\in\Sigma$ of $\sigma$ and $c\in C$, let 
$$J(T, C, c) = \prod\limits_{i=0}^{\#C-1}j(T,\sigma^i(c))$$

The \emph{conjugacy invariant} $J^k(T)$ is the collection of conjugacy classes in $J^k$ of $J(T,C,c)$ for all cycles $C\in \Sigma$. Note that $J^k(T)$ does not depend on the choices of $c\in C$. 

\subsection{Reduction to simple deformations}

Let $T_0\in DC(\eta,\theta,\sigma)$ be an almost Roth i.e.t. acting on $I=(u_0(T_0),u_d(T_0))$. 

Fix $r\geq 3$, set $d^* := (g-1)(2r+1)+s+g-\mu$, let $\ell\geq0$ be an integer, consider a neighborhood $V=[-t_0,t_0]^{\ell+d^*}$ of the origin $0\in\mathbb{R}^{\ell+d^*}$, and write $t\in V$ as $t=(t', t'')$ with $t'\in[-t_0,t_0]^{\ell}$ and $t''\in[-t_0,t_0]^{d^*}$. 

Denote by $(T_t)_{t\in V}$ a family of generalized i.e.t. on $I$ of class $C^{r+3}$ with the same combinatorial data of $T_0$. We assume that $t\mapsto T_t$ is $C^1$ and the conjugacy invariant of $T_t$ in $J^{r+3}$ is trivial for all $t\in V$. 

In this context, the derivative of $T_t$ with respect to $t$ is a linear map $\Delta T$ from $\mathbb{R}^{\ell+d^*}$ to $C^{r+1+BV}_{\partial}(\sqcup_{\alpha\in\A} I^t_{\alpha}(T_0))$. By composing $\Delta T$ with the projection $L_1(r+1)$ from Theorem \ref{t.A} in order to get 
$$\overline{\Delta T}:=L_1(r+1)\circ \Delta T: \mathbb{R}^{\ell+d^*}\to \Gamma_{\partial}(r+1)/\Gamma_{T_0}(r+1)$$
Recall that the dimension of $\Gamma_{\partial}(r+1)/\Gamma_{T_0}(r+1)$ is $(g-1)(2r+1) + g-\mu+1 = d^*-s+1$. We suppose that $(T_t)_{t\in V}$ satisfies the transversality conditions (Tr1) and (Tr2) from page 1607 of \cite{MMY12}. 

The precise version of Theorem \ref{t.B} is:

\begin{theorem}\label{t.51} Under the hypothesis above, there are $t_1\leq t_0$ and a neighborhood $W$ of the identity in $\textrm{Diff}^r(\overline{I(T_0)})$ such that: 
\begin{itemize}
\item for each $t'\in[-t_1,t_1]^{\ell}$, there are unique $t''\in[-t_1,t_1]^{d^*}$ and $h_{t'}\in W$ so that $T_{(t',t'')} = h_{t'}\circ T_0\circ h_{t'}^{-1}$;  
\item the maps $t'\mapsto t'':=\theta(t')$ and $t'\mapsto h_{t'}$ are $C^1$ and $D\theta|_{t'=0}=0$. 
\end{itemize}
\end{theorem} 

As it turns out, we can reduce Theorem \ref{t.51} to the case of \emph{simple deformations}: more concretely, Proposition 5.3 of \cite{MMY12} allows to pick $t_2<t_0$ and a $C^1$ family $(\widetilde{h}_t)_{t\in[-t_2,t_2]^{\ell+d^*}}$ in $\textrm{Diff}^{r+3}(\overline{I(T)})$ such that $\widetilde{T}_t = \widetilde{h}_t\circ T_t\circ \widetilde{h}_t^{-1}$ has singularities at $u_i(T_0)$, $0<i<d$, and $\widetilde{T}_t=T_0$ near the endpoints of $I_{\alpha}^t(T_0)$ for all $\alpha\in\A$. 

\subsection{End of proof of Theorem \ref{t.B}} It is not hard to see that the analysis of Herman's Schwarzian derivative trick performed in Subsections 6.1 to 6.6 of \cite{MMY12} can be performed verbatim after noticing that the subspace $\Gamma_u$ with $\Gamma_{\partial}(r-2) = \Gamma_{T_0}(r-2)\oplus\Gamma_u\oplus\mathbb{R}1$ has dimension 
$$\textrm{dim } \Gamma_u = \textrm{dim }\Gamma_{\partial}(r-2)/\Gamma_{T_0}(r-2) - 1 = (g-1)(2r-4)+1-\mu$$ 

At this point, our task consists in checking that the equations (6.1), (6.2), (6.3) and (6.4) in Subsection 6.7 of \cite{MMY12} provide $d^*+2$ independent conditions. In this direction, we observe that (6.1) provides $(d-1)$ equations, (6.4) gives $\textrm{dim }\Gamma_u = (g-1)(2r-4)+1-\mu$ equations, and (6.2) and (6.3) imposes $2(2g-1)$ equations (\emph{independently} on whether $(\alpha_t, R)$ and $({}_b\alpha, L)$ belong to the same cycle of $\sigma$ or not). Hence, the total number of independent equations is 
$$(d-1) + (g-1)(2r-4)+1-\mu + 2(2g-1) = d^*+2$$ 

\subsection*{Acknowledgements} This article grew out of several conversations in Pisa during a visit of G.~Forni and Carlos~Matheus to S.~Marmi. In particular, it is a pleasure for us to thank Scuola Normale Superiore and Centro Ennio De Giorgi for their hospitality during the preparation of this text. G.~Forni acknowledges the support
of the NSF Grant DMS 1600687. S.~Marmi acknowledges the support by UniCredit  R\&D division under the project {\it Dynamics and Information Research Institute}.




\begin{thebibliography}{99}

\bibitem{AV07}  A.~Avila and M.~Viana,  Simplicity of Lyapunov spectra: proof of the Zorich-Kontsevich conjecture, Acta Math. 198 (1) (2007), 1--56.

\bibitem{Fo97} G.~Forni, \emph{Solutions of the cohomology equation for area-preserving flows
on compact Riemann surfaces}, Ann. of Math. 146 (1997), 295--344.

\bibitem{Fo02} G.~Forni, \emph{Deviation of ergodic averages for area-preserving flows on surfaces of higher genus}, Ann. of Math. (2) 155 (2002), 1--103.

\bibitem{Fo07} G.~Forni, \emph{Sobolev regularity of solutions of the cohomological equation}, preprint,  arXiv:0707.0940v2.

\bibitem{Fo11} G.~Forni, \emph{A geometric criterion for the nonuniform hyperbolicity of the Kontsevich--Zorich cocycle},  J. Mod. Dynam.  5(2) (2011), 355--395.

\bibitem{FMZ11} G.~Forni, C.~Matheus and A.~Zorich, \emph{Square-tiled cyclic covers}, J. Mod. Dyn. 2 (2011), 285--318. 

\bibitem{FMZ14} G.~Forni, C.~Matheus and A.~Zorich, \emph{Zero Lyapunov exponents of the Hodge bundle}, Comment. Math. Helv. 89 (2) (2014) 489--535.

\bibitem{Ke} M.~Keane, \emph{Interval exchange transformations}, Math. Z. 141 (1975), 25--31. 

\bibitem{MMY05} S.~Marmi, P.~Moussa and J.-C.~Yoccoz, \emph{The cohomological equation for Roth-type interval exchange maps},  J. Amer. Math. Soc.  18 (2005), 823--872.

\bibitem{MMY12} S.~Marmi, P.~Moussa and J.-C.~Yoccoz, \emph{Linearization of generalized interval exchange maps},  Ann. Math. 176 (2012),  1583--1646.

\bibitem{MY16} S.~Marmi and J.-C.~Yoccoz, \emph{H\"older regularity of the solutions of the cohomological equation for Roth type interval exchange maps}, Comm. Math. Phys. 344 (2016), 117--139.

\bibitem{MMY15}   C.~Matheus, M.~M\"oller and J.-C.~Yoccoz, \emph{A criterion for the simplicity of the Lyapunov spectrum of square-tiled surfaces},  Invent. math. 202 (1) (2015) 333--425.

\end{thebibliography}
\end{document}